\documentclass[journal]{IEEEtran} 

\usepackage[TU]{fontenc} 
\usepackage{amsfonts}
\usepackage{amsthm}
\usepackage{amssymb}
\usepackage{graphicx}
\usepackage{subcaption}
\usepackage{lipsum}
\usepackage{multicol} 
\usepackage{mathtools}
\usepackage{multirow}
\usepackage{threeparttable}
\usepackage{gensymb}
\usepackage[table,pdftex,hyperref,svgnames]{xcolor} 
\usepackage{url} 
\usepackage{color, colortbl}  

\usepackage{array}
\usepackage[normalem]{ulem}

\usepackage{hyperref}%
\hypersetup{colorlinks=true, linkcolor=blue, breaklinks=true, urlcolor=blue}

\usepackage{multicol} 
\usepackage{mathtools}
\usepackage{version,xspace}
\usepackage{url,doi}
\newcommand{\until}[1]{\{1,\dots, #1\}}

\DeclareMathOperator{\sgn}{sgn}

\newcommand\oprocendsymbol{\hbox{$\triangle$}}
\newcommand\oprocend{\relax\ifmmode\else\unskip\hfill\fi\oprocendsymbol}

\DeclareSymbolFont{bbold}{U}{bbold}{m}{n}
\DeclareSymbolFontAlphabet{\mathbbold}{bbold}
\newcommand{\vect}[1]{\mathbbold{#1}}

\renewcommand{\circ}{\odot}
\newcommand{\scirc}{\raise1pt\hbox{$\,\scriptstyle\circ\,$}}
\newcommand{\real}{\mathbb{R}}

 \newcommand{\cl}{\mathrm{cl}}

\newtheorem{theorem}{Theorem}
\newtheorem{lemma}{Lemma} 
\newtheorem{remark}{Remark}

\newtheorem{definition}{Definition}
\newtheorem{assumption}{Assumption}
\newtheorem{problem}{Problem}

\DeclareMathOperator{\diag}{diag}

\DeclareMathOperator{\adj}{adj}
\renewcommand{\top}{\mathsf{T}}
 
\begin{document}
	
	\title{Convex Optimization \\ of the Basic Reproduction Number}

	\author{Kevin D. Smith, \IEEEmembership{Student Member, IEEE}, and Francesco Bullo, 
	\IEEEmembership{Fellow, IEEE}%
		\thanks{Kevin D. Smith and Francesco Bullo are with the Center for 
		Control, 
		Dynamical Systems and Computation, UC Santa Barbara, CA 93106-5070, USA. {\tt 
		\{kevinsmith,bullo\}@ucsb.edu}}
		\thanks{This work was supported in part by the U.S.\ Defense Threat 
		Reduction Agency under grant HDTRA1-19-1-0017 and by the AFOSR grant
FA9550-22-1-0059. The authors would like to thank 
		Aaron Bagheri in the UCSB Department of Mathematics for insightful comments and 
		discussion.}}
	
	\maketitle
	
	\begin{abstract}
	  The basic reproduction number $R_0$ is a fundamental quantity in
      epidemiological modeling, reflecting the typical number of
      secondary infections that arise from a single infected
      individual. While $R_0$ is widely known to scientists,
      policymakers, and the general public, it has received
      comparatively little attention in the controls community. This
      note provides two novel characterizations of $R_0$: a stability
      characterization and a geometric program characterization. The
      geometric program characterization allows us to write
      $R_0$-constrained and budget-constrained optimal resource
      allocation problems as geometric programs, which are easily transformed into 
      convex optimization problems. We apply these programs to allocating 
      vaccines and antidotes in numerical examples, finding that targeting 
      $R_0$ instead of the spectral abscissa of the Jacobian matrix (a common target 
      in the controls literature) leads to qualitatively different solutions. 
	\end{abstract}
	

	\IEEEpeerreviewmaketitle
	
	\section{Introduction}
	
	Perhaps the most important parameter in an epidemic is the basic
        reproduction number.  This number, denoted $R_0$, is the number of
        secondary infections that arise from a typical infected individual
        within an otherwise completely susceptible population.  $R_0$ is a
        widely-known term, especially since 2020, when articles with ``$R_0$''
        in the title ran in mainstream publications like \textit{The New York
          Times} and \textit{The Wall Street Journal}. Since $R_0$ is an
        intuitive and widely-known quantity, one might also expect it to appear
        frequently in the controls literature on epidemics, but this is not the
        case.
	
	Instead, the literature tends to focus on two other major approaches to
        epidemic control.  First, in the \emph{optimal control framework},
        parameters or control inputs are chosen to minimize some cost function
        integrated along the model trajectory \cite{RER-RL-CAG:09, SL-GC-CCC:10,
          MH-FB-VMP:21, VLJS-IRM:21}. These trajectories seldom admit
        closed-form solutions, so this approach generally requires
        model-specific analysis and numerical solutions of Pontryagin's
        conditions \cite{RER-RL-CAG:09, SL-GC-CCC:10}, potentially large-scale
        optimization to embed discrete-time dynamics \cite{MH-FB-VMP:21}, or
        linearization and a discount factor to ensure convergence
        \cite{VLJS-IRM:21}. The second major approach is the \emph{spectral
          optimization framework}, in which resources are allocated to minimize
        the spectral abscissa of the model's Jacobian matrix about some
        disease-free equilibrium \cite{VMP-MZ-CE-AJ-GJP:14, JAT-SR-YW:17,
          CN-VMP-GJP:17, VSM-AB-KM:18, ARH-JG-PEP:21}.  If the Jacobian is
        stable, then the abscissa represents the rate at which the trajectory
        converges to this equilibrium, so minimizing the (negative) abscissa
        leads to a faster-decaying epidemic. Spectral optimization is based on a
        linear approximation of the model, but it is nonetheless an appealing
        framework for resource allocation, since the spectral abscissa can be
        directly evaluated from model parameters (without computing a
        trajectory).
	
	The spectral abscissa is closely related to $R_0$. They are equivalent threshold 
	parameters for whether the epidemic spreads or decays: in compartmental epidemic 
	models (under
    reasonable assumptions), the epidemic enters an exponential growth phase
    if and only if the abscissa is positive, if and only if $R_0 > 1$
    \cite{PvdD-JW:02}. Furthermore, intuitively, both quantities reflect the
    \textit{rate} at which the epidemic spreads or decays. But it is important to
    note that the abscissa and $R_0$ are different quantities. In fact, through 
    proper choice of infection and recovery rates in the 
    Kermack-McKendrick SIR model, one can achieve \textit{any} pair of
    values for the abscissa $\alpha$ and reproduction number $R_0$ 
    such that $R_0 > 0$ and $\sgn(\alpha) = \sgn(R_0 - 1)$. Thus, while the 
    intuition for these two quantities is similar, minimizing the abscissa will 
    generally lead to a different allocation of resources than minimizing $R_0$ 
    directly.
        
    To our knowledge, there is no work in the literature that focuses on directly 
    minimizing or constraining $R_0$ in the resource allocation problem. Motivated 
    by the ubiquity of $R_0$ in epidemiology and its popularity in the public discourse 
    around COVID-19, this note provides theoretical foundations to fill in this gap.
        
	
	\paragraph*{Contributions}
	
	We propose a modification of the spectral optimization framework
        to operate on $R_0$ instead of on the spectral abscissa. We offer three
        primary contributions:
	\begin{enumerate}
	\item We provide two novel characterizations of $R_0$ in compartmental
          epidemic models. One characterization relates $R_0$ to the stability
          of perturbations to the Jacobian matrix, and the other expresses $R_0$
          as a geometric program, which can be transformed into a convex
          optimization problem.
                
	\item We define two $R_0$-based optimal resource allocation problems:
          the $R_0$-constrained allocation problem, which identifies the
          lowest-cost allocation to restrict $R_0$ below a given upper bound;
          and the budget-constrained allocation problem, which minimizes $R_0$
          with a limited allowance for resource cost. We provide a geometric
          programming transcription for both of these problems, allowing them to
          be solved efficiently with off-the-shelf software.
                
	\item We present numerical results based on a county-level
          multi-group SEIR model in California, parameterized using real-world
          cell phone mobility data. The experiments study the allocation of
          vaccines and antidotes, a classical problem in spectral
          optimization. We explain and emphasize the differences between the allocations 
          based on $R_0$ and the corresponding allocations based on the abscissa.
	\end{enumerate}   

	\paragraph*{Organization}

	Section \ref{sect:prelim} introduces the general family of compartmental epidemic 
	models that we consider (\S\ref{sect:compartmental}), formally defines $R_0$ 
	(\S\ref{sect:R0}), briefly reviews geometric programming (\S\ref{sect:geometric}), 
	and states three key lemmas about Metzler and Hurwitz matrices 
	(\S\ref{sect:lemmata}). Section \ref{sect:opt} presents our main theoretical results, 
	including the two new characterizations of $R_0$ (\S\ref{sect:char}), and the two 
	$R_0$-based optimal resource allocation problems and their geometric program 
	transcriptions (\S\ref{sect:alloc}). Finally, Section \ref{sect:numerical} presents 
	the numerical experiments. 

	\paragraph*{Notation}
	
	The matrix $A \in \real^{n \times n}$ is \textit{Metzler} if all
        its off-diagonal entries are non-negative and is \textit{Hurwitz}
        if all its eigenvalues have negative real part. Let $\rho(A)$
        denote the spectral radius of $A$. Given $A \in \real^{n \times
          n}$, let $\diag(A)$ denote the vector in $\real^n$ composed of
        the diagonal elements of $A$. Given $x \in \real^n$, let $\diag(x)$
        denote the diagonal matrix whose diagonal is $x$. Thus
        $\diag(\diag(x)) = x$, and $\diag(\diag(A))$ is a copy of $A$ with
        all off-diagonal entries set to zero. Given a set $S$, we write
        $\cl(S)$ to denote the closure of $S$.
	
	\section{Preliminaries}
	\label{sect:prelim}
	
	\subsection{Compartmental Epidemic Models}
	\label{sect:compartmental}
	Compartmental models are a general and widely-used family of epidemic models that 
	divide a population into compartments based on disease state and 
	other demographic factors. This paper focuses on deterministic epidemic models, in 
	which the number of individuals in each compartment is governed by a system 
	of differential equations. Perhaps the most well-known example is 
	Kermack and McKendrick's SIR model, which has three compartments (susceptible, 
	infected, and recovered), but compartmental models can be arbitrarily complex to 
	capture nuances in the spread of infection between different parts of the population 
	in different disease states. Compartmental models are frequently based on
	an underlying stochastic model, such that the state variables approximate the 
	expected number of individuals in each compartment. 
	
	We consider the general compartmental model in \cite{PvdD-JW:02},
        with $n$ infected compartments and $m$ non-infected
        compartments. The components of this model are as follows. Let $x
        \in \real^n$ be the expected numbers of individuals in each
        infected compartment, and let $y \in \real^m$ be the expected
        numbers of non-infected individuals. The resulting dynamics is
        \begin{subequations}\label{eqs}
	\begin{align}
	  \dot x &= f(x, y) + v(x, y) \label{eq:x} \\
	  \dot y &= g(x, y) \label{eq:y}
	\end{align} 
        \end{subequations}
	where $f$, $v$, and $g$ are continuously differentiable and defined on non-negative 
	domains. The dynamics of the infected 
	subsystem are decomposed into two vector fields $f$ and 
	$v$, where $f$ contains the rates at which 
	new infections appear, and $v$ contains rates of transitions that do not correspond 
	to new infections. For example, if infected individuals must pass through a latent 
	disease state before entering an active infectious state (as in the SEIR model), then 
	$f$ captures new infections as they appear in the latent state, while transitions 
	from latent to active infections are contained in $v$, since the latter are not 
	altogether new infections. This explicit separation of rates corresponding 
	to new infections from all other transitions is crucial to the computation of $R_0$, 
	and it reflects extra physical interpretation that cannot be inferred from the 
	expression for $\dot x$ alone.
	
	\begin{assumption}[Regularity of $f$, $v$, and $g$] \label{ass:regularity}
		The vector fields $f$, $v$, and $g$ have the following properties:
		\begin{enumerate}
			\item $f(x, y) \ge \vect 0_n$ for all $x$ and $y$; \label{cond:nonneg}
			\item \label{cond:zero} $f(\vect 0_n, y) = \vect 0_n$ and $v(\vect 0_n, y) = 
			\vect 0_n$ for all $y$; 
			\item \label{cond:dry-v} for all $x$, $y$, and $i$, $x_i = 0$ implies that 
			$v_i(x, y) \ge 0$;
			\item \label{cond:dry-g} for all $x$, $y$, and $j$, $y_j = 0$ implies that 
			$g_j(x, y) \ge 0$.
		\end{enumerate}
	\end{assumption}

	\noindent 
	Assumption \ref{ass:regularity} collects weak conditions that are
        obvious from the physical interpretations of $f$, $v$, and
        $g$. Condition~\ref{cond:nonneg} follows from the interpretation of $f$
        as a rate at which new infections are
        created. Condition~\ref{cond:zero} ensures that no individuals can
        transfer into or out of an infected compartment (through new
        infections or otherwise) if the population is completely free of
        disease; thus every disease-free state is an equilibrium of
        \eqref{eq:x}. Finally, conditions~\ref{cond:dry-v} and~\ref{cond:dry-g}
        reflect the fact that individuals cannot transition out from an
        empty compartment.
	
	We also assume that \eqref{eqs} admit a disease-free equilibrium
        point $(\vect 0_n, y^*)$ that is locally asymptotically stable
        \textit{in the absence of new infections}. That is, if new
        infections are ``switched off'' by dropping the vector field $f$
        from the dynamics, then the population will return to $(\vect 0_n,
        y^*)$ even if a small number of infected individuals are
        introduced.
	\begin{assumption}[Existence of a Stable Equilibrium] \label{ass:dfe}
		There exists $y^* \ge \vect 0_m$ such that $g(\vect 0_n, y^*) = \vect 0_m$ and 
		the following Jacobian matrix  is Hurwitz:
		\[
			D \begin{bmatrix} v(\vect 0_n, y^*) \\ g(\vect 0_n, y^*) \end{bmatrix}
			= \begin{bmatrix}
				D_x v(\vect 0_n, y^*) & D_y v(\vect 0_n, y^*) \\
				D_x g(\vect 0_n, y^*) & D_y g(\vect 0_n, y^*)
			\end{bmatrix}.
		\]
	\end{assumption} 

	\noindent
	The point $(\vect 0_n, y^*)$ satisfying Assumption \ref{ass:dfe} is not necessarily 
	unique, and while it is also an equilibrium point of the full model, it may be 
	unstable when $f$ is no longer ignored. 
	
	Under Assumptions \ref{ass:regularity} and \ref{ass:dfe}, linearizing the dynamics of 
	\eqref{eq:x} about $(\vect 0_n, y^*)$ decouples them from $y$, and we obtain
	\begin{equation}
		\dot x = (F + V) x \label{eq:linear}
	\end{equation}
	where $F = D_x f(\vect 0_n, y^*)$ is non-negative and $V = D_x v(\vect 0_n, y^*)$ is 
	Hurwitz and Metzler. We refer the reader to \cite[Lemma 1]{PvdD-JW:02} for the 
	details of this linearization.  
	
	\subsection{Basic Reproduction Numbers}
	\label{sect:R0}
	
	The basic reproduction number is well-known in epidemiology as the
        typical number of secondary infections that arise from a single
        infected individual, within an otherwise completely susceptible
        population. Diekmann, Heesterbeek, and Metz~\cite{OD-JAPH-JAJM:90}
        introduced the next generation operator to compute this quantity in
        general models with structured populations. This approach was later
        applied by van den Driessche and Watmough \cite{PvdD-JW:02}
        specifically to the compartmental model~\eqref{eqs}.

    \begin{definition}
      For a compartmental epidemic model~\eqref{eq:x}-\eqref{eq:y}
      satisfying Assumptions~\ref{ass:regularity} and~\ref{ass:dfe} and
      with linearization~\eqref{eq:linear} about $(\vect 0_n, y^*)$, the \textit{basic 
      reproduction number} is
	  \begin{equation} \label{eq:char-ngm}
	    R_0 = \rho(F V^{-1}).
	  \end{equation} 
    \end{definition}

	\noindent
	We refer the reader to \cite[\S 2]{OD-JAPH-JAJM:90} and \cite[\S 3]{PvdD-JW:02} for 
	derivations of \eqref{eq:char-ngm} from the epidemiological definition of $R_0$.
	
	\subsection{Geometric Programming}
	\label{sect:geometric}
	
	Geometric programs are a family of generally non-convex
        optimization problems that can be transformed into convex
        optimization problems by a change of variables. Geometric programs
        enjoy a multitude of applications in engineering and control
        theory, including the design of optimal positive systems
        \cite{MO-MK-JL:20}, a problem which is closely related to the
        resource allocation considered in this note. We refer the reader to
        \cite{SB-SJK-LV-AH:07} as a standard introduction to geometric
        programming and briefly introduce the key concepts in what follows.
	
	A \textit{monomial function} is a map $\real^n_{> 0} \to \real_{>
          0}$ of the form $f(x) = c x_1^{b_1} x_2^{b_2} \cdots x_n^{b_n}$,
        where $c > 0$ and $b_i \in \real$. A \textit{posynomial function}
        is a sum of monomial functions. Note that posynomials are closed
        under addition and multiplication, and that a posynomial divided by
        a monomial is a posynomial. Given a posynomial function $f_0$, a
        set of posynomial functions $f_i$, $i\in\until{m}$, and a set of
        monomial functions $g_i$, $i\in\until{p}$, a geometric program in
        standard form is:
	\[
		\begin{array}{rl}
			\text{minimize}: & f_0(x) \\
			\text{variables}: & x > \vect 0_n \\
			\text{subject to}: & f_i(x) \le 1, \; i\in\until{m}\\
			& g_i(x) = 1, \; i\in\until{p}
		\end{array}
	\]
	The problem becomes convex after the change of variables $x_i = e^{y_i}$.  
	Off-the-shelf software is available for geometric programs, including the CVX package 
	in MATLAB \cite{MG-SB:11-cvx}.

	\subsection{Properties of Hurwitz and Metzler Matrices}
	\label{sect:lemmata}
	
	We now reproduce three lemmas regarding properties of Metzler and Hurwitz matrices 
	that will be necessary for our main results. The first lemma is a standard result 
	characterizing the stability of Metzler matrices 
        (see \cite[Theorem~10.14]{FB:22}):
        
	\begin{lemma}[Metzler Hurwitz Lemma]
	  \label{thm:metzler-hurwitz:2}\index{Theorem!Metzler Hurwitz}
	  Let $M\in\real^{n\times{n}}$ be a Metzler matrix. The following are equivalent:
	  \begin{enumerate}
	  \item\label{fact:Metzler-Hurwitz:1-bis} $M$ is Hurwitz,
	  \item\label{fact:Metzler-Hurwitz:2-bis} $M$ is invertible and $-M^{-1}\geq0$, 
	    and
	  \item\label{fact:rantzer:1} there exists $w > \vect 0_n$ such that $M w < \vect 0_n$. 
	  \end{enumerate} 
	\end{lemma} 

	\noindent
	We borrow the next two results from \cite{PvdD-JW:02}; the first is a
        slight restatement of \cite[Lemma 5]{PvdD-JW:02}, so we do not include a
        proof.
	\begin{lemma}[Properties of Hurwitz and Metzler Matrices] \label{lem:exercise}
		Let $H, M \in \real^{n \times n}$ be Metzler matrices, such that $H$ is Hurwitz 
		and $-M H^{-1}$ is Metzler. The following are equivalent:
		\begin{enumerate}
			\item $M$ is Hurwitz, and
			\item $-M H^{-1}$ is Hurwitz. 
		\end{enumerate}
	\end{lemma}
	
	\noindent
	The second result is abstracted from the proof of \cite[Theorem 2]{PvdD-JW:02} and we 
	include a self-contained proof.
	
	\begin{lemma}[Stability of Perturbed Metzler Matrices] \label{lem:stability}
		Let $H \in \real^{n \times n}$ be Metzler and Hurwitz, and let $E \in \real_{\ge 
			0}^{n \times n}$ be a non-negative perturbation matrix. The following are 
		equivalent:
		\begin{enumerate}
			\item $H + E$ is Hurwitz, and
			\item $\rho(-E H^{-1}) < 1$.
		\end{enumerate}
	\end{lemma}
	
	\begin{proof}
		Let $A = -(H + E)H^{-1} = -(I_n + E H^{-1})$. 
		Note that $A$ is Metzler, since $-H^{-1} \ge 0$ by Lemma 
		\ref{thm:metzler-hurwitz:2}, so $-E H^{-1} \ge 0$. Then by Lemma 
		\ref{lem:exercise}, $H + E$ is 
		Hurwitz if and only if $A$ is Hurwitz. If $\rho(-E 
		H^{-1}) < 1$, then $A$ is clearly Hurwitz. But if $\rho(-E H^{-1}) \ge 1$, then 
		$A$ is not Hurwitz: since $-E H^{-1} \ge 0$, the Perron-Frobenius theorem 
		guarantees that 
		its dominant eigenvalue is real and non-negative, so $-(I_n + E H^{-1})$ has an 
		eigenvalue with non-negative real part.
	\end{proof}

	
	\section{Optimization Framework for $R_0$}
	\label{sect:opt}
	
	\subsection{Geometric Program for $R_0$}
	\label{sect:char}
	The main theoretical result of this paper is the following theorem, which provides two novel 
	characterizations of $R_0$:
	
	\begin{theorem}[Characterizations of $R_0$] \label{thm:char}
		Consider the linearized epidemic dynamics \eqref{eq:linear} with $F \in 
		\real_{\ge 0}^{n \times n}$ and $V \in \real^{n \times n}$ Hurwitz and Metzler. 
		Write $V = V_{od} - V_d$, where $V_d \ge 0$ is diagonal and $V_{od} \ge 0$ has 
		zero diagonal. The following are characterizations of the basic 
		reproduction number:
		\begin{enumerate} 
			\item \textit{Stability characterization}:
			\begin{equation} \label{eq:char-stab}
				R_0 = \inf_{r > 0} \{r : F + rV ~\text{is Hurwitz} \}
			\end{equation}
			\item \textit{Geometric program characterization}:
			\begin{equation} \label{eq:char-geo}
				R_0 = \inf_{\substack{r > 0 \\ w > \vect 0_n}} \! \left\{
				r : \diag(r V_d w)^{-1} (F + r V_{od}) w \le \vect 1_n
				\right\}
			\end{equation}
		\end{enumerate} 
	\end{theorem}

	\begin{proof}
		To prove that \eqref{eq:char-stab} follows from \eqref{eq:char-ngm}, we compute
		\begin{align*}
			\inf_{r > 0} \{r: F + rV ~\text{is Hurwitz}\}
			&= \inf_{r > 0} \{r : \rho(F(rV)^{-1}) < 1 \} \\
			&= \inf_{r > 0} \{r : \rho(F V^{-1}) < r \} \\
			&= \inf_{r > 0} \{r : R_0 < r\} =  R_0,
		\end{align*}
		where the first step follows from Lemma \ref{lem:stability}. We now use 
		\eqref{eq:char-stab} to prove \eqref{eq:char-geo}. Let $W = \{w > 
		\vect 0_n : V w < \vect 0_n\}$ and $\hat W = \{w > \vect 0_n : V w \le \vect 
		0_n\}$. By Lemma \ref{lem:relax} (in Appendix \ref{app:relax}),
		\begin{align*}
			R_0 &= \inf \{r > 0 : F + rV~\text{is Hurwitz}\} \\
			&= \inf \{r > 0 : \exists w \in W ~\text{s.t.}~ (F + r V)w < \vect 
			0_n\} \\
			&= \inf \{r > 0 : \exists w \in \hat W ~\text{s.t.}~ (F + r V)w \le \vect 
			0_n\} \\
			&= \inf_{r > 0, \; w > \vect 0_n} \left\{r : (F + r V)w \le \vect 0_n \right\}
		\end{align*}  
		In the last step, we note that the $V w \le \vect 
		0_n$ constraint is implied by $(F + r V)w \le \vect 0_n$, so we are free to 
		remove it. Manipulating the 
		$(F + r V)w \le \vect 0_n$ constraint into the standard form for geometric 
		programming yields \eqref{eq:char-geo}. 
	\end{proof} 

	\begin{remark}[Degenerate Cases, Pt. I] \label{rem:not-attained}
	The infimum in \eqref{eq:char-geo} is not always attained. For example, if 
	$F = \begin{bmatrix}0&0\\1&1\end{bmatrix}$ and $V = -\begin{bmatrix}1&0\\0&1 
	\end{bmatrix}$, then $
		R_0 = \inf_{\substack{r> 0 \\ w > \vect 0_2}} \left\{
			r : r \ge \frac{w_1 + w_2}{w_2}
		\right\} = 1$.
	But there is no feasible point $w > \vect 0_2$ that satisfies the inequality 
	constraint with $r = 1$. Thus, in general, we cannot replace the infimum in 
	\eqref{eq:char-geo} with a minimum. 
	\end{remark}  

 
	\subsection{Optimal Resource Allocation}
	\label{sect:alloc}
	
	The geometric program characterization \eqref{eq:char-geo} sets us up to
        efficiently optimize model parameters to either minimize or constrain
        $R_0$. In a manner analogous to \cite{VMP-MZ-CE-AJ-GJP:14}, we consider
        two forms of the resource allocation problem: \textit{$R_0$-constrained
          allocation}, and \textit{budget-constrained allocation}. In both forms
        of the resource allocation problem, we suppose that the model parameters
        $F$, $V_{od}$, and $V_d$ depend on a vector of ``resources'' $\theta \ge
        \vect 0_k$, and that the cost of a particular allocation of resources is
        given by a cost function $c(\theta)$. Furthermore, the resources must
        satisfy some collection of constraints $h(\theta) \le \vect 1_q$. The
        dependence on $\theta$ must obey the following conditions:
	
	\begin{assumption}[Resource Dependence] \label{ass:geometric}
		The resource dependence of the parameters $F(\theta)$, $V_{od}(\theta)$, 
		$V_d(\theta)$, $c(\theta)$, and $h(\theta)$ have the following properties:
		\begin{enumerate}
			\item \label{cond:posy} $F(\theta)$, $V_{od}(\theta)$, $c(\theta)$, and 
			$h(\theta)$ are element-wise posynomial functions;
			\item \label{cond:mono} $V_d(\theta)$ is an element-wise monomial function; 
			and
			\item \label{cond:domain} the set of feasible allocations $\{\theta \ge \vect 
			0_k : h(\theta) \le \vect 1_q\}$ is bounded, and if $\theta$ is in this set, 
			then $V_{od}(\theta) - V_d(\theta)$ is Hurwitz. 
		\end{enumerate}  
	\end{assumption} 

	\noindent
	Conditions~\ref{cond:posy} and \ref{cond:mono} are necessary to transcribe the allocation 
	problem as a geometric program, while condition~\ref{cond:domain} ensures that the matrix 
	parameters $F$, $V_{od}$, and $V_d$ satisfy the antecedent of Theorem \ref{thm:char} 
	for any feasible allocation. Condition~\ref{cond:domain} also ensures the feasible 
	$\theta$ are confined to a compact set. Under these assumptions, for all $\theta \in 
	h_{\le}^{-1}(\vect 1_q)$, the resource dependence of $R_0$ can be written as
	\begin{equation} \label{eq:r0-theta}
		R_0(\theta) = \rho\left( 
		F(\theta) (V_{od}(\theta) - V_d(\theta))^{-1}\right).
	\end{equation}
	
	Additional resources will typically reduce the rate of new infections or increase 
	the rate at which existing infections are removed. This property is not included in 
	Assumption \ref{ass:geometric}, since it is not needed for any of the results in this 
	section. However, if this property is true, then it is useful (albeit unsurprising) 
	to note that $R_0(\theta)$ is weakly decreasing in $\theta$.
	
	\begin{lemma}[Monotonicity] \label{lem:mono}
	  Suppose that $F(\theta)$, $V_{od}(\theta)$, and $V_d(\theta)$ satisfy
          Assumption~\ref{ass:geometric}. If additionally $F(\theta)$ and
          $V_{od}(\theta)$ are non-increasing and $V_d(\theta)$ is
          non-decreasing in $\theta$, then for $\theta, \theta' \in
          h_{\le}^{-1}(\vect 1_q)$ with $\theta' \ge \theta$, we have
          $R_0(\theta') \le R_0(\theta)$.
	\end{lemma} 
	
	\begin{proof}    
		Let $\theta' \ge \theta$. Since $0 \le F(\theta') \le F(\theta)$, 
		$0 \le V_{od}(\theta') \le V_{od}(\theta)$, and $V_d(\theta') \ge V_d(\theta) \ge 
		0$, we can write $F(\theta) = F(\theta') + \Delta F$ and $V(\theta) = V(\theta') 
		+ \Delta V(\theta)$ for some matrices $\Delta F, \Delta V \ge 0$. Then
		\begin{align*}
			V^{-1}(\theta) - V^{-1}(\theta') &= (V(\theta') + \Delta V)^{-1} - 
			V^{-1}(\theta') \\
			&= -(V(\theta') + \Delta V)^{-1} (\Delta V) V^{-1}(\theta') \\
			&\le 0 
		\end{align*}
		where the last inequality follows from Lemma \ref{thm:metzler-hurwitz:2}, since 
		$V(\theta)$ and $V(\theta')$ 
		are Hurwitz and Metzler, and thus $V^{-1}(\theta) \le 0$ and $V^{-1}(\theta') \le 
		0$. Then
		\begin{align*}
			-F(\theta) V^{-1}(\theta) &= -(F(\theta') + \Delta F) V^{-1}(\theta) \\
			&\ge -F(\theta') V^{-1}(\theta) \\
			&\ge -F(\theta') V^{-1}(\theta')
		\end{align*}
		Since $-F(\theta) V^{-1}(\theta) \ge 0$ and $-F(\theta') V^{-1}(\theta') \ge 0$, 
		we are guaranteed that
		\begin{align*}
			R_0(\theta) &= \rho(-F(\theta) V^{-1}(\theta)) \\
			&\ge \rho(-F(\theta') V^{-1}(\theta')) \\
			&= R_0(\theta')
		\end{align*}
		since the spectral radius is weakly increasing in the elements of a 
		non-negative matrix \cite[Theorem 8.1.18]{RAH-CRJ:12}. 
	\end{proof}
	
	We now define the two optimal allocation problems. In the $R_0$-constrained 
	allocation problem, we identify the cheapest allocation of 
	resources to ensure that $R_0 \le r_{\rm max}$, where $r_{\rm max} > 0$ is some 
	arbitrary threshold. In the budget-constrained allocation problem, some 
	budget $c_{\rm max} > 0$ is available to spend on resources, and we would like to 
	deploy these limited resources to minimize $R_0$. 
	
	\begin{definition}[Optimal Allocation Problems]
		Let $F(\theta)$, $V_{od}(\theta)$, $V_d(\theta)$, $c(\theta)$, and $h(\theta)$ 
		satisfy Assumption \ref{ass:geometric}. We define the following optimization 
		problems:
		\begin{enumerate}
			\item Given $r_{\rm max} > 0$, we say that $\theta^*$ is an optimal 
			$R_0$-constrained allocation if $\theta^*$ is a minimizer of 
			\begin{equation} \label{prob:r0-constrained-explicit}
				\min_{\theta \ge \vect 0_k} \left\{
				c(\theta) : h(\theta) \le \vect 1_q ~\text{and}~ R_0(\theta) \le 
				r_{\rm max}
				\right\}
			\end{equation}
			\item Given $c_{\rm max} > 0$, we say that $\theta^*$ is an optimal 
			budget-constrained allocation if $\theta^*$ is a minimizer of 
			\begin{equation} \label{prob:budget-constrained-explicit}
				\min_{\theta \ge \vect 0_k} \left\{
				R_0(\theta) : h(\theta) \le \vect 1_q ~\text{and}~ c(\theta) \le 
				c_{\rm max}
				\right\}
			\end{equation}
		\end{enumerate}
	\end{definition}  

	\noindent
	Assumption \ref{ass:geometric} ensures that $R_0(\theta)$ in 
	\eqref{eq:r0-theta} is well-defined over the feasible sets; furthermore,
	$R_0(\theta)$ is continuous, since the matrix inverse and spectral radius are 
	continuous functions of the matrix elements. Thus the feasible sets are compact, so 
	the minima of both problems exist. 
	
	Using Theorem \ref{thm:char}, we can construct a pair of geometric programs to solve 
	for optimal $R_0$-constrained and budget-constrained allocations. For notational 
	convenience, we define a map $p: \real_{> 0} \times \real_{> 0}^n \times \real_{\ge 
	0}^k \to \real_{> 0}^n$ by
	\begin{equation}
		p(r, w, \theta) = \diag(r V_d(\theta) w)^{-1} (F(\theta) + r V_{od}(\theta)) w 
	\end{equation}
	Under Assumption \ref{ass:geometric}, $p(r, w, \theta)$ is posynomial, so the 
	following are geometric programs:
	
	\begin{problem}[$R_0$-Constrained Allocation GP] \label{prob:r0-constrained}
		Given $r_{\rm max} > 0$ and a tolerance parameter $\tau \ge 0$:
		\begin{equation*} \label{eq:r0-constrained}
			\begin{array}{rl}
				\text{minimize}: & c(\theta) \\
				\text{variables}: & r > 0, \; w > \vect 0_n, \; \theta > \vect 0_k \\
				\text{subject to}: & p(r, w, \theta) \le \vect 1_n \\ 
				& h(\theta) \le \vect 1_q \\
				& r \le r_{\rm max} + \tau
			\end{array}
		\end{equation*} 
	\end{problem}   
	
	\begin{problem}[Budget-Constrained Allocation GP] \label{prob:budget-constrained}
		Given $c_{\rm max} > 0$:
		\begin{equation*}
			\begin{array}{rl}
				\text{minimize}: & r \\
				\text{variables}: & r > 0, \; w > \vect 0_n, \; \theta > \vect 0_k \\
				\text{subject to}: & p(r, w, \theta) \le \vect 1_n \\
				& h(\theta) \le \vect 1_q \\
				& c(\theta) \le c_{\rm max}
			\end{array}
		\end{equation*} 
	\end{problem}  

	\begin{theorem}[Geometric Program Transcription] \label{thm:transcription}
		Let $\theta^* \ge \vect 0_k$, $r_{\rm max} > 0$, and $c_{\rm max} > 0$. Let 
		$\mathcal F_1(\tau)$ for $\tau > 0$ and $\mathcal F_2$ be the sets of 
		feasible points $(r, w, \theta)$ for Problems \ref{prob:r0-constrained} and 
		\ref{prob:budget-constrained}. The following are true: 
		\begin{enumerate} 
			\item \label{prop:r0-constrained} $\theta^*$ is an optimal $R_0$-constrained 
			allocation if and only if the infimum of Problem \ref{prob:r0-constrained} 
			converges to $c(\theta^*)$ as $\tau \to 0_+$ and there exists $r^*, 
			w^*$ such that $(r^*, w^*, \theta^*) \in \cl(\mathcal F_1(\tau))$ for 
			all $\tau > 0$. 
			\item \label{prop:budget-constrained} $\theta^*$ is an optimal 
			budget-constrained allocation if and only if 
			$R_0(\theta^*)$ is the infimum of Problem \ref{prob:budget-constrained} and 
			there exists $r^*, w^*$ such that $(r^*, w^*, \theta^*) \in 
			\cl(\mathcal F_2)$.
		\end{enumerate}  
	\end{theorem}  
	
	\noindent
	See Appendix \ref{app:transcription} for the proof.  
	
	We note that Problem \ref{prob:r0-constrained} is an arbitrarily accurate 
	approximation of the $R_0$-constrained allocation problem, controlled by the 
	parameter $\tau \ge 0$. This approximation is necessary due to the closed 
	inequality constraint on $R_0$ and the representation of $R_0$ by the infimum 
	in \eqref{eq:char-geo}, which is not always attained: 
	
	\begin{remark}[Degenerate Cases, Pt. II]
		In some cases, Problem \ref{prob:r0-constrained} may be infeasible when 
		$\tau = 0$, for example, if $F(\theta) = F$ and $V(\theta) = V$ are the 
		matrices defined 
		in Remark~\ref{rem:not-attained} and $r_{\rm max} = 1$. Fortunately, the feasible 
		set is nonempty for all $\tau > 0$, so we	can still consider the limit of 
		solutions to Problem \ref{prob:r0-constrained} as $\tau \to 0_+$.
		This feasibility problem arises due to the constraint on $R_0$, so it is not an 
		issue in Problem \ref{prob:budget-constrained}.  
	\end{remark} 

	In practice, the issue of an empty feasible set is not of significant concern, since 
	numerical optimization already has inherently limited precision. We suggest solving 
	Problem \ref{prob:r0-constrained} with $\tau = 0$ (and only using a small 
	positive value if the solver reports primal infeasibility). 

	\section{Numerical Examples} 
	\label{sect:numerical} 
	
	In the following experiments, we compare $R_0$-minimizing allocations with 
	abscissa-minimizing allocations.
	The code used to generate these results is 
	available online.\footnote{The MATLAB script and functions used to generate these 
	results is available at      
	\url{https://www.mathworks.com/matlabcentral/fileexchange/99354-geometric-programs-for-r0}.
       Running the code requires an installation of CVX 2.2 and the MOSEK solver.}
	
	\subsection{Epidemic Model}
	
	We adopt a standard multigroup SEIR model (with vital dynamics) for an epidemic in 
	the state of California, 
	where each group corresponds to one of the state's $n = 58$ counties. The SEIR model 
	has two infected states (exposed and infectious) and two non-infected states 
	(susceptible and recovered). Letting $s, e, z, r \in \real^n_{\ge 0}$ denote the 
	expected number of people in each group and disease state, the model dynamics for 
	each group $i\in\until{n}$ are
	\begin{align*}
		\dot s_i &= -\beta_i s_i \sum_{j = 1}^n a_{ij} z_j & \dot z_i &= \gamma_i e_i - 
		\delta_i z_i \\
		\dot e_i &= \beta_i s_i \sum_{j = 1}^n a_{ij} z_j - \gamma_i e_i & \dot r_i &= 
		\delta_i z_i 
	\end{align*}
	It is clear that the model has a disease-free equilibrium $(s_0, \vect 0_n, 
	\vect 0_n, \vect 0_n)$. Linearizing about this point, we obtain 
	\[
		\begin{bmatrix} \dot e \\ \dot z \end{bmatrix} \approx
		\begin{bmatrix}
			-\diag(\gamma) & \diag(\beta) \diag(s_0) A \\
			\diag(\gamma) & -\diag(\delta)
		\end{bmatrix} \begin{bmatrix} e \\ z \end{bmatrix}.
	\]
	Because the $\diag(\beta) \diag(s_0) A$ term is the only one corresponding to 
	the creation of new infections, we decompose this Jacobian into the two matrices
	\[
		F \!=\! \begin{bmatrix}0 & \diag(\beta) \diag(s_0) A  \\ 0 & 0 \end{bmatrix}, \;
	 V \!=\! \begin{bmatrix} -\diag(\gamma) & 0 \\ \!\!\diag(\gamma) & \!\!-\diag(\delta) 
	 \end{bmatrix},
	\]
	where $F$ is non-negative and $V$ is Hurwitz and Metzler.  

	
	The model requires a matrix of inter-group contact rates $A \in \real_{\ge 0}^{n 
	\times n}$, which we estimated using data from 
	SafeGraph.\footnote{\href{https://www.safegraph.com}{SafeGraph} is a data company 
	that aggregates anonymized location data from numerous applications in order to 
	provide insights about physical places, via the SafeGraph Community. To enhance 
	privacy, SafeGraph excludes census block group information if fewer than two devices 
	visited an establishment in a month from a given census block group.} In particular, 
	we used the \href{https://docs.safegraph.com/docs/social-distancing-metrics}{Social 
	Distancing Metrics} dataset to estimate a matrix $P \in \real^{n \times n}$, where 
	$p_{ij}$ is the daily fraction of people from county $i$ who visited county $j$, 
	averaged over each day in 2020. Then $(P P^\top)_{ij}$ approximates the probability 
	that two random individuals from counties $i$ and $j$ are co-located in the same 
	county on a given day. We set $A = \alpha P P^\top$, where the 
	scalar $\alpha = 2.3667 \times 10^{-7}$ was chosen to ensure $R_0 = 2.5$ when 
	$\beta = 0.1$, $\gamma = 0.2$, and $\delta = 0.1$. Note that $\alpha$ is 
	always multiplied by $\beta$, so the only effect of this scalar is to allow us to 
	work with round numbers for $\beta$ and $R_0$.
	
	The remaining model parameters are the transmission rates $\beta > \vect 0_n$, 
	incubation rates $\gamma > \vect 0_n$, and recovery rates $\delta > \vect 0_n$ for 
	each group. We used uniform model parameters across each group for simplicity. We 
	generated 2,000 different models by choosing $\beta$, $\gamma$, and $\delta$ for each 
	of 10 ($\gamma$ and $\delta$) or 20 ($\beta$) evenly-spaced values in the range 
	$[0.025, 0.5]$, $[0.05, 0.5]$, and $[0.05, 
	0.5]$, respectively. The $\gamma$ and $\delta$ range was chosen to allow for a wide 
	range of mean incubation and recovery times (between 2 and 20 days), while 
	the $\beta$ range was coarsely tuned so that the models have a wide 
	but realistic range of pre-intervention $R_0$ (95\% between 0.23 and 19.38).

	\subsection{Optimal Allocation of Pharmaceuticals}

	We consider the following optimal resource allocation scenario from 
	\cite{VMP-MZ-CE-AJ-GJP:14}, in which there are two types of pharmaceutical 
	interventions: vaccines, which reduce the local 
	transmission rates $\beta_i$; and antidotes, which increase the local recovery rates 
	$\delta_i$. By allocating vaccines to patch $i$, we can optimize the local 
	transmission rate within a range $\beta_i \in [\underline \beta_i, \; \overline 
	\beta_i]$, where $\overline \beta_i \ge \underline \beta_i > 0$. The cost of this 
	vaccine allocation is, for all $i$, 
	\begin{equation}
		f_i(\beta_i) = \frac{\beta_i^{-1} - \overline 
		\beta_i^{-1}}{\underline\beta_i^{-1} - 
		\overline \beta_i^{-1}}.
		\label{eq:f}
	\end{equation}
	Note that the most aggressive allocation has a cost of $f_i(\underline \beta_i) = 1$, 
	while allocation of no vaccines at all has a cost $f_i(\overline \beta_i) = 0$. The 
	form of \eqref{eq:f} ensures diminishing returns in the investment of vaccines at 
	each patch. Similarly, by allocating antidotes to patch $i$, the local recovery rate 
	can be 
	optimized in the range $\delta_i \in [\underline \delta_i, \; \overline \delta_i]$, 
	with $\overline \delta_i \ge \underline \delta_i > 0$. The cost of the antidote 
	allocation is, for all $i$,
	\begin{equation}
		g_i(\delta_i) = \frac{(\tilde \delta_i - \delta_i)^{-1} - (\tilde \delta_i - 
		\underline \delta_i)^{-1}}{(\tilde \delta_i - \overline \delta_i)^{-1} - (\tilde 
		\delta_i - \underline \delta_i)^{-1}},
		\label{eq:g}
	\end{equation}
	where the parameters $\tilde \delta_i > \overline \delta_i$ control the shape of the 
	cost curve. The total cost, summing over the local costs of vaccines and antidotes 
	over all patches, is constrained by a budget $c_{\rm max}$. 

	In order to perform budget-constrained resource allocation, we 
	must encode the following budget constraint in the standard form for geometric 
	programming:
	\[
		\sum_{i=1}^n f_i(\beta_i) + g_i(\delta_i) \le c_{\rm max}
	\]
	Since $g_i$ have non-posynomial dependence on $\delta_i$, we replace 
	$1 - \delta_i$ with auxiliary variables $\eta_i$, constrained by $\tilde \delta_i 
	- \overline \delta_i \le \eta_i \le \tilde \delta_i - \underline \delta_i$. Then 
	the posynomial budget constraint is
	\begin{equation}
		\sum_{i = 1}^n \frac{\kappa^{-1}\beta_i^{-1}}{\underline\beta_i^{-1} 
			- \overline \beta_i^{-1}} + \frac{\kappa^{-1} \eta_i^{-1}}{(\tilde 
			\delta_i - 
			\overline \delta_i)^{-1} - (\tilde \delta_i - \underline \delta_i)^{-1}} 
			\le 1
		\label{eq:budget-constraint}
	\end{equation}
	where we define a positive constant
	\[
	\kappa = c_{\rm max} + \sum_{i=1}^n \frac{\overline \beta_i^{-1}}{\underline 
		\beta_i^{-1} - \overline \beta_i^{-1}} + \frac{(\tilde \delta_i - \underline 
		\delta_i)^{-1}}{(\tilde \delta_i 
		- \overline \delta_i)^{-1} - (\tilde \delta_i - \underline \delta_i)^{-1}}
	\] 
	Altogether, the resource vector is $\theta^\top = \begin{bmatrix} \beta^\top & 
	\eta^\top \end{bmatrix}$, and the constraints are
	$\underline \beta_i \le \beta_i \le \overline \beta_i$, $\tilde 
	\delta_i - \overline \delta_i \le \eta_i \le \tilde \delta_i - \underline 
	\delta_i$, and \eqref{eq:budget-constraint}. 
	
	For each experiment, we selected cost parameters based on the pre-intervention SEIR 
	model parameters $\beta$ and $\delta$. Since pharmaceuticals and vaccines never 
	increase the transmission rate or decrease the recovery rate, we set
	$\overline \beta_i = \beta_i$ and $\underline \delta_i = \delta_i$. We chose
	$\underline \beta_i = 0.1 \beta_i$ and $\overline \delta_i = 2 \delta_i$ to reflect a 
	90\% reduction in transmissibility and 50\% reduction in mean recovery time at 
	maximum investment, and we selected $\tilde \delta_i = 2$ so that $\tilde \delta_i > 
	\overline \delta_i$.
	
	\subsection{Results and Discussion}  

	We first set a budget of $c_{\rm max} = 0.1$ and performed budget-constrained 
	resource allocation to minimize $R_0$ and the abscissa for each of the 2,000 models. 
	We then simulated the nonlinear post-intervention dynamics for both the 
	$R_0$-minimized and abscissa-minimized models until convergence.
	
	\begin{figure}
		\centering
		\includegraphics[width=0.49\linewidth]{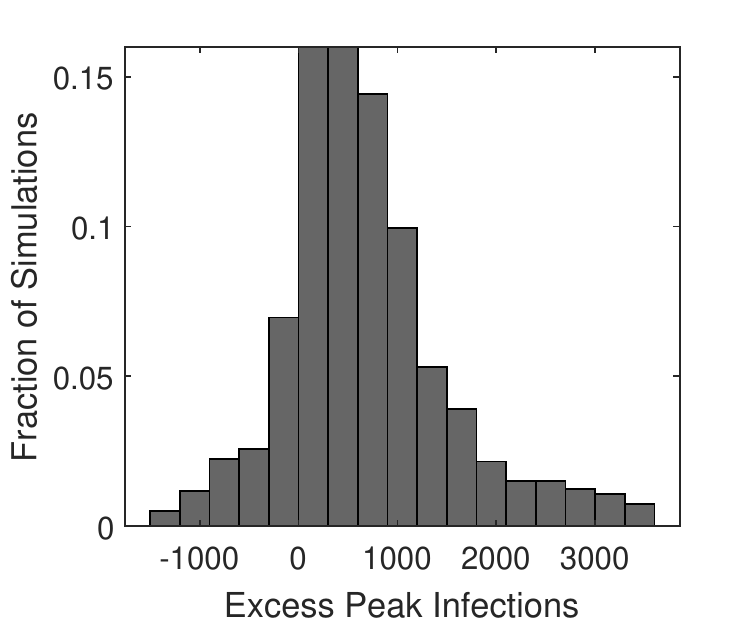}
		\includegraphics[width=0.49\linewidth]{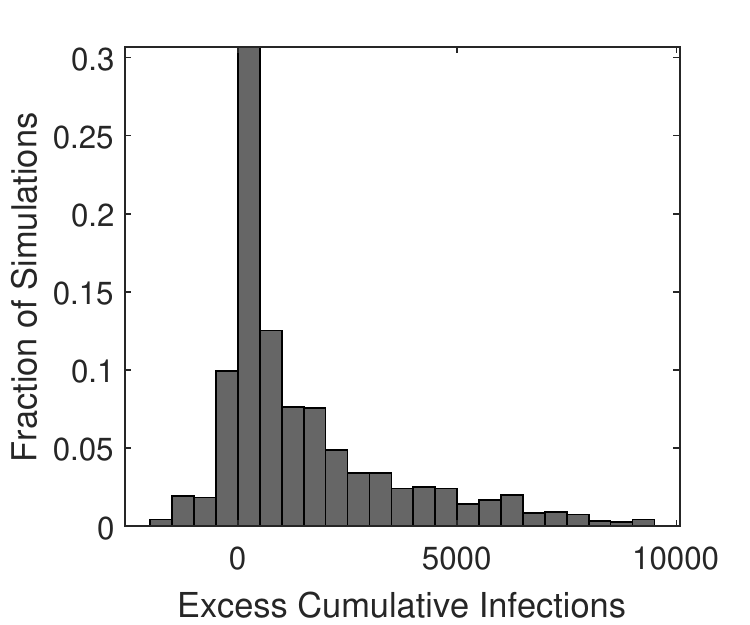}
		\caption{Comparison of peak (left) and cumulative (right) infections from 
		minimizing $R_0$ vs. the abscissa, in models with an initial exponential growth 
		phase. Both histograms show the 
		distribution of how many more infections resulted in the absicssa-minimizing 
		scenario vs. the $R_0$-minimizing scenario.}
		\label{fig:histograms}
	\end{figure}
	
	In 1,270 models, both the $R_0$-minimized and abscissa-minimized models had $R_0 > 
	1$, so the number of infected individuals experienced an initial exponential growth 
	phase before peaking and decaying. Figure \ref{fig:histograms} (left) compares the 
	number of active infections at the peak between the $R_0$-minimized and 
	abscissa-minimized trajectories. In 1,068 (84.1\%) of these models, minimizing $R_0$ 
	led to a smaller peak than minimizing the absicssa. Similarly, Figure 
	\ref{fig:histograms} (right) compares the number of cumulative infections at the end 
	of the simulation. Minimizing $R_0$ resulted in fewer cumulative cases in 1,056 
	(83.1\%) in the example models. In the remaining models, one or both of the 
	$R_0$-minimizing or abscissa-minimizing 
	allocations led to $R_0 < 1$, so the trajectory immediately decays toward a 
	disease-free equilibrium. It is not meaningful to compare peaks in these models; 
	however, in 96.4\% of them, minimizing $R_0$ resulted in fewer cumulative infections.

	\begin{figure}
		\centering
		\includegraphics[width=0.49\linewidth]{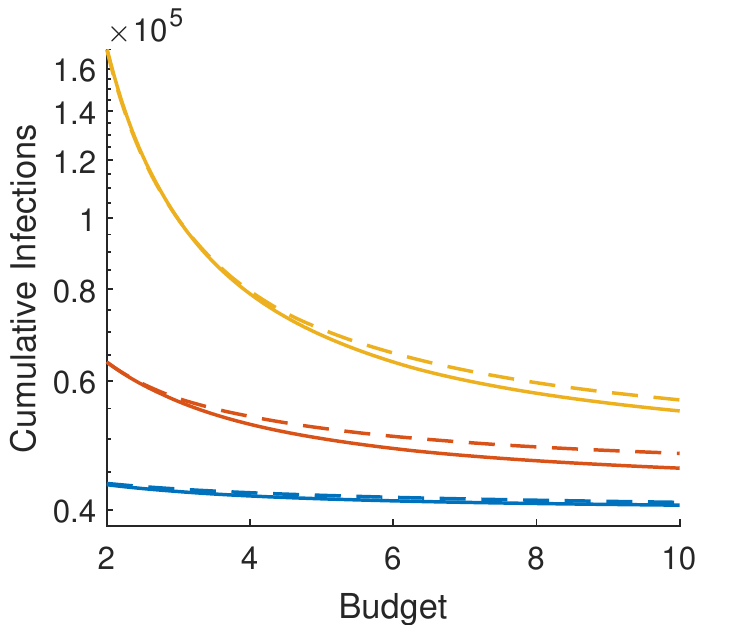}%
		\includegraphics[width=0.49\linewidth]{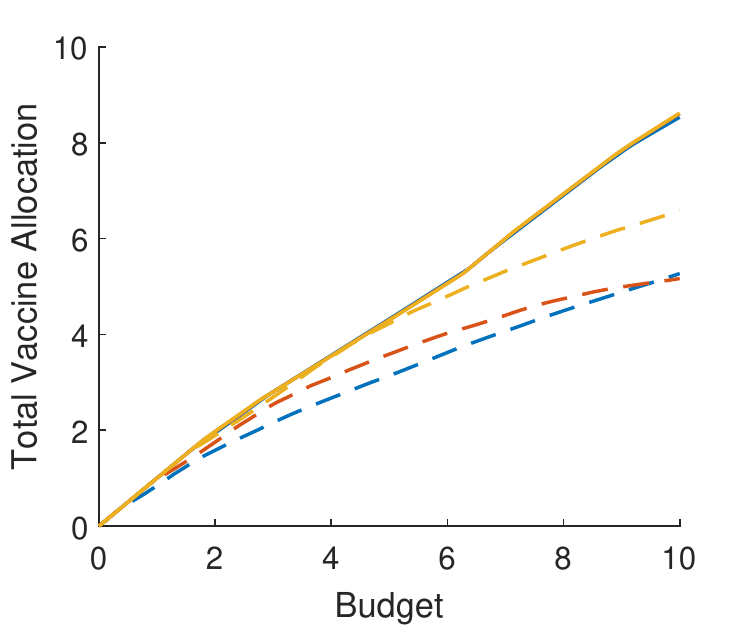}
		\includegraphics[width=0.9\linewidth, clip, trim={2in 5.2in 2in 
		5.2in}]{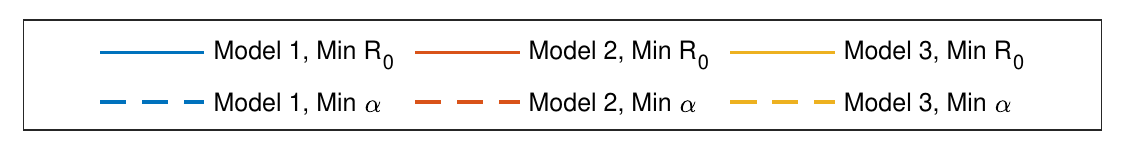}
		\caption{Cumulative infections (left) and total budget allocated to vaccines 
		(right) for three models, given various budgets. Solid lines correspond to 
		post-intervention models minimizing $R_0$, while dashed lines reflect minimizing 
		the abscissa. In these examples, minimizing $R_0$ results in fewer cumulative 
		infections and a greater fraction of the budget allocated to vaccines.}
		\label{fig:budgets}
	\end{figure}
	
	Next, we selected three particular models to examine the allocations under various 
	budgets. We chose a low-$R_0$ model ($\beta = 0.05$, $\gamma = 0.2$, $\delta = 0.2$; 
	$R_0 = 0.625$), a mid-$R_0$ model ($\beta = 0.1$, $\gamma = 0.2$, $\delta = 0.1$; 
	$R_0 = 2.5$), and a high-$R_0$ model ($\beta = 0.15$, $\gamma = 0.2$, $\delta = 
	0.075$; $R_0 = 5.0$), and we repeated the budget-constrained allocations at various 
	budgets. Figure \ref{fig:budgets} (left) plots the 
	cumulative infections for the post-intervention models. Cumulative infections in the 
	$R_0$-minimized and abscissa-minimized models are very similar at low budgets, but 
	past a budget of 2, minimizing the $R_0$ leads to a modest decrease in 
	cumulative infections when compared to minimizing the abscissa. (It is not meaningful 
	to plot the peak infections, since $R_0 < 1$ in all post-intervention models with 
	budgets above 2.) Figure \ref{fig:budgets} (right) illustrates a difference in 
	allocation strategies between the two targets, as minimizing $R_0$ 
	results in a larger share of the budget spent on vaccines. 
	
	\section{Conclusion}
	
	In this note, we have established a new formula for the basic reproduction number of 
	a compartmental epidemic model.  We then applied this formula to resource 
	allocation problems that minimize or constrain $R_0$, transcribing these 
	problems as geometric programs, and we have provided numerical experiments to 
	highlight that targeting $R_0$ instead of the abscissa 
	can result in qualitatively different solutions. Our results show that 
	$R_0$ can be a superior target for controlling cumulative and peak infections; 
	however, more work is needed to identify for which models and parameter ranges this 
	is the case. The possible applications of our optimization framework are broad, 
	since it applies to a general class of epidemic models and cost functions.
	Policymakers should be aware of the limitations of optimal resource 
	allocation: models (and linear models in particular) have limited accuracy, and 
	mathematics does not address the complex social factors of epidemic response. 
	Nonetheless, we believe that this work and its future extensions---coupled with 
	judicious choices of models and cost functions---can provide useful 
	insight for epidemic preparedness and response.  
	
	\appendices 
	
	\section{A Relaxing Lemma}
	\label{app:relax}  
	
	\begin{lemma}[A Relaxing Lemma] \label{lem:relax}
		Let $V \in \real^{n \times n}$ be a Metzler and Hurwitz matrix, and let $F \ne 0$ 
		be a non-negative matrix of the same shape.
		Let $W = \{w > \vect 0_n : V w < \vect 0_n\}$, $\hat W = \{w > \vect 0_n : V w 
		\le \vect 0_n\}$, $R_0 = \inf \{r > 0 : \exists w \in W ~\text{s.t.}~ (F + r V)w 
		< \vect 0_n\}$, and $\hat R_0 = \inf \{r > 0 : \exists w \in \hat W ~\text{s.t.}~ 
		(F + r V)w \le \vect 0_n\}$. Then $R_0 = \hat R_0$.  
	\end{lemma}
	
	\begin{proof}
		It is obvious that $\hat R_0 \le R_0$, so we need only show that $\hat R_0 \ge 
		R_0$. Before we embark on this task, we will construct useful expressions for 
		$\hat R_0$ and $R_0$. Let $I$ be the (possibly empty) set of 
		indices for which the $i$th row of $F$ is zero: $F^{(i)} = \vect 0_n$. 
		For any $w \in \hat W$, observe that $\{r > 0 : (F + r V)w \le \vect 0_n\}$ is 
		non-empty if and only if $(V w)_i = 0$ implies that $i \in I$. Thus, we define 
		\[
			\bar W = \{w > \vect 0_n : V w \le \vect 0_n ~\text{and}~(V w)_i < 0 
			~\text{for all}~ i \in I^c\}
		\]
		where $W \subset \bar W \subset \hat W$. Then we can write
		\begin{align*}
			\hat R_0 &= \inf\left(
			\bigcup_{w \in \hat W} \{r > 0 : (F + r V) w \le \vect 0_n \}
			\right) \\
			&= \inf\left( \bigcup_{w \in \bar W} \{r > 0 : (F + r V) w \le \vect 
			0_n\} \right) \\
			&= \inf_{w \in \bar W} \left( \inf\{r > 0 : (F + r V) w \le \vect 
			0_n\}\right) = \inf \hat{\mathcal R}
		\end{align*}
		where $\hat{\mathcal R} = \{r^*(w) : w \in \bar W\}$, and $r^*: \bar W \to 
		\real_{\ge 0}$ is the map defined by
		\[
			r^*(w) = \inf\{r > 0 : (F + r V) w \le \vect 0_n\}, \;\; \forall w \in \bar 
			W
		\]		
		It is straightforward to solve for $r^*(w)$:
		\[
			r^*(w) = \max_{i \in I^c} \left\{
			\frac{(F w)_i}{|V w|_i}
			\right\}, \;\; \forall w \in \bar W
		\]
		Similar to $\hat R_0$, we have the following expression for $R_0$: 
		\begin{align*}
			R_0 &= \inf\left(\bigcup_{w \in W} \{r > 0 : (F + r V)w < \vect 0_n\}\right) 
			\\
			&= \inf_{w \in W} \left( \inf\{r > 0 : (F + r V) w < \vect 0_n\} \right) 
			\\ 
			&= \inf_{w \in W} \left( \min\{r > 0 : (F + r V) w \le \vect 0_n \} \right)
			= \inf \mathcal R
		\end{align*}
		where $\mathcal R = \{r^*(w) : w \in W\}$. 
		
		The remainder of the proof is to show 
		that $R_0$ is a lower bound on $\hat{\mathcal R}$. Let $\hat r \in \hat{\mathcal 
		R}$, so that $\hat r > 0$ and $(F + \hat r V) \hat 
		w \le \vect 0_n$ for some $\hat w \in \bar W$. Let $x > \vect 0_n$ such that $V 
		x < \vect 0_n$ (which must exist because $V$ is Hurwitz), and for all $t \ge 0$, 
		let $w(t) = \hat w + t x$. We can also show that
		\[
			\left|\frac{(F w(t))_i}{|V w(t)|_i} - \frac{(F \hat w)_i}{|V \hat w|_i} 
			\right| \le \kappa_i t, \;\; \forall t \ge 0 ~\text{and}~ \forall i \in I^c
		\]
		where
		\[
		\kappa_i = \frac{1}{|V \hat w|_i} \left( (F x)_i + \frac{(F \hat w)_i |V 
			x|_i}{|V \hat w|_i }\right) > 0, \;\; \forall i \in I^c
		\]
		Then for all $t > 0$,
		\begin{align*}
			\left|r^*(w(t)) - r^*(\hat w)\right|
			&= \left| \max_{i \in I^c} \left\{
			\frac{(F w(t))_i}{|V w(t)|_i}
			\right\} - \max_{j \in I^c} \left\{
			\frac{(F \hat w)_j}{|V \hat w|_j}
			\right\} \right| \\
			&\le \left( \max_{i \in I^c} \kappa_i \right) t 
		\end{align*} 
		Thus, given any $\epsilon > 0$, we can choose $t < \epsilon \left( \max_{i 
			\in I^c} \kappa_i \right)^{-1}$ to ensure that $|r^*(w(t)) - r^*(\hat w)| < 
		\epsilon$. Because $w(t) \in W$ for all $t > 0$, it is the case that $r^*(w(t)) 
		\in \mathcal R$ for all $t > 0$, so that every open ball around $r^*(\hat w)$ 
		contains a point in $\mathcal R$. Then $r^*(\hat w) \in \cl(\mathcal R)$, which 
		implies that $r^*(\hat w) \ge R_0$. But $r^*(\hat w) \le \hat r$, and $\hat r$ 
		was chosen arbitrarily from $\hat{\mathcal R}$, so $R_0$ is a lower bound on 
		$\hat{\mathcal R}$. But $\hat R_0$ is the greatest such lower bound, so we 
		conclude that $\hat R_0 \ge R_0$.  
	\end{proof}
	
	\section{Proof of Theorem \ref{thm:transcription}}
	\label{app:transcription}
	
	Let $\mathcal G_1, \mathcal G_2$ be the sets of feasible points $\theta$ for 
	\eqref{prob:r0-constrained-explicit} and \eqref{prob:budget-constrained-explicit}, 
	respectively. We define a Metzler matrix
	\begin{equation} \label{eq:M}
		M(r, \theta) = F(\theta) + r V_{od}(\theta) - r V_d(\theta) 
	\end{equation}
	Since the determinant of $M(r, \theta)$ is a polynomial in $r$ of degree $n$, for 
	some scalars $a_1, a_2, \dots, a_n \in \mathbb C$, we can write $|M(r, \theta)| = (r 
	- a_1)(r - a_2) \cdots (r - a_n) $. Due to \eqref{eq:char-stab} in Theorem 
	\ref{thm:char}, $M(R_0(\theta), \theta)$ must 
	be singular, so $R_0(\theta)$ is a root; then we can
	assign $a_1, a_2, \dots, a_\ell = R_0(\theta)$ up to some multiplicity $\ell$. Define 
	a ``pseudo-determinant'' $\mu(r, \theta) = (r - a_{\ell + 1}) \cdots (r - a_n)$ as 
	the product of the remaining factors, which is real and nonzero for all $r \ge 
	R_0(\theta)$. Then 
	\[
		M^{-1}(r, \theta) = \frac{\adj(M(r, \theta))}{(r - R_0(\theta))^\ell \mu(r, 
		\theta)}, \;\; \forall r > R_0(\theta)
	\]
	Now, pick $z > \vect 0_n$ arbitrarily, and define
	\begin{align}
		w(r, \theta) &= -(r - R_0(\theta))^\ell M^{-1}(r, \theta) z, \; \; \forall r > 
		R_0(\theta) \\
		w^*(\theta) &= \!\!\!\!\! \lim_{r \to R_0(\theta^*)^+} \!\!\!\! w(r, \theta) = 
		-\left(\frac{\adj(M(R_0(\theta), \theta))}{\mu(R_0(\theta), \theta)} \right) z 
		\label{eq:w}
	\end{align} 
	For any $r > R_0(\theta)$, \eqref{eq:char-stab} in Theorem \ref{thm:char} implies 
	that $M(r, \theta)$ is Hurwitz, so $-M^{-1}(r, \theta) \ge 0$, and thus $w(r, \theta) 
	> \vect 0_n$. Furthermore, $M(r, \theta) w(r, \theta) < \vect 0_n$, so expanding 
	$M(r, \theta)$ with \eqref{eq:M} and re-arranging, we obtain $p(r, w(r, \theta), 
	\theta) < \vect 1_n$. We now use $w^*(\theta)$ to formally establish relationships 
	between the feasible sets of both pairs of optimization problems:

	\begin{lemma}[Relating the Feasible Sets] \label{lem:feasible} 
		For each $\tau > 0$, let $\Theta_1(\tau) \subset \real^k$ be the set 
		of $\theta$ such that $(r, w, \theta) \in \cl(\mathcal F_1(\tau))$ for some 
		$r, w$. Similarly, let $\Theta_2 \subset \real^k$ be the set of $\theta$ such 
		that $(r, w, \theta) \in \cl(\mathcal F_2)$ for some $r, w$. The following are 
		true:
		\begin{enumerate} 
			\item \label{prop:f3} $\theta \in \mathcal G_1 \!\!\implies\!\! (R_0(\theta), 
			w^*(\theta), \theta) \in \cl(\mathcal F_1(\tau))$ for 
			all $\tau > 0$, 
			\item \label{prop:f4} $\theta \in \mathcal G_2 \implies (R_0(\theta), 
			w^*(\theta), \theta) \in \cl(\mathcal F_2)$,
			\item \label{prop:f1} $\mathcal G_1 = \bigcap_{\tau > 0}
			\Theta_1(\delta)$, and
			\item \label{prop:f2} $\mathcal G_2 = \Theta_2$.
		\end{enumerate}  
	\end{lemma}

	\begin{proof}
		To prove \ref{prop:f3}, let $\theta \in \mathcal G_1$, so $h(\theta) \le 
		\vect 1_q$ and $R_0(\theta) \le r_{\rm max}$. Fix any $\tau > 0$, and let 
		$\epsilon > 0$. By \eqref{eq:w}, we can choose $r > R_0(\theta)$ such that 
		$||w(r, \theta) - w^*(\theta)|| < \epsilon$ and $|r - R_0(\theta)| < 
		\min\{\tau, \epsilon\}$. Since $p(r, w(r,\theta), \theta) \le \vect 1_q$ 
		and $r < R_0(\theta) + \tau \le r_{\rm max} + \tau$, we have $(r, 
		w(r, \theta), \theta) \in \mathcal F_1(\tau)$, so every 
		neighborhood of $(R_0(\theta), w^*(\theta), \theta)$ (by choice of $\epsilon$) 
		contains a point in $\mathcal F_1(\tau)$. We prove \ref{prop:f4} by a 
		similar argument (without $\tau$), noting that $\theta \in \mathcal G_2$ 
		implies $c(\theta) \le c_{\rm max}$.  
		
		To prove \ref{prop:f1}, we note that \ref{prop:f3} implies that $\mathcal G_1 
		\subseteq \bigcap_{\tau > 0} \Theta_1(\tau)$. If $\theta \in 
		\Theta_1(\tau)$ for all $\tau > 0$, then $h(\theta) \le \vect 1_q$ 
		and $R_0(\theta) \le r_{\rm max} + \tau$ for all $\tau > 0$, which 
		implies $R_0(\theta) \le r_{\rm max}$, and thus $\theta \in \mathcal G_1$. Hence 
		$\mathcal G_1 \supseteq \bigcap_{\tau > 0} \Theta_1(\tau)$ as well. 
		Statement \ref{prop:f2} follows from a similar argument.  
	\end{proof}  
	
	\begin{proof}[Proof of Theorem \ref{thm:transcription}]
		In order to prove \ref{prop:r0-constrained}, we first define $c^*(\delta)$ as the 
		infimum of Problem \ref{prob:r0-constrained} for all $\tau > 0$, and we 
		define $c^*$ as the minimum cost of the $R_0$-constrained allocation problem. 
		Noting that $\Theta_1(\tau)$ are nested downward as $\tau \to 
		0$: 
		\[
			c^* \!=\! \min \mathcal G_1
			\!=\! \min \bigcap_{\tau > 0} \Theta_1(\tau)
			\!=\! \lim_{\tau \to 0^+} \min \Theta_1(\tau)
			\!=\! \lim_{\tau \to 0^+} c^*(\tau)
		\]
		The second step is due to Lemma \ref{lem:feasible}, and the third step is a 
		general property of intersections of nested sets. Let $\theta^*$ be an optimal 
		$R_0$-constrained allocation. Then $\theta^* \in 
		\mathcal G_1$, so by Lemma \ref{lem:feasible},
		$(R_0(\theta^*), w^*(\theta^*), \theta^*) \in \cl(\mathcal F_1(\tau))$ for 
		all $\tau > 0$, and we have shown that $c^*(\tau) \to c^* = 
		c(\theta^*)$ as $\tau \to 0_+$. On the 
		other hand, if there exist $r^*, w^*$ such that $(r^*, w^*, \theta^*) \in 
		\cl(\mathcal F_1(\tau))$ for all $\tau > 0$, then Lemma 
		\ref{lem:feasible} 
		guarantees $\theta^* \in \mathcal G$, and $c^*(\tau) \to c(\theta^*)$ 
		implies that $c(\theta^*) = c^*$. 
		
		We now prove \ref{prop:budget-constrained}. Let $\theta^*$ be an optimal 
		budget-constrained allocation. Then $\theta^* \in 
		\mathcal G_2$, so Lemma \ref{lem:feasible} implies that $(R_0(\theta^*), 
		w^*(\theta^*), \theta^*) \in \cl(\mathcal F_2)$. Consider any other point 
		$(r, w, \theta) \in \cl(\mathcal F_2)$, and note that Lemma \ref{lem:feasible} 
		also implies $\theta \in \mathcal G_2$, so that $R_0(\theta^*) \le R_0(\theta)$. 
		But $R_0(\theta) \le r$ by \eqref{eq:char-geo}, so $R_0(\theta^*) \le r$. Thus 
		$R_0(\theta^*)$ is the min value of $r$ over $\cl(\mathcal F_2)$.
		
		Finally, suppose that $(R_0(\theta^*), w^*(\theta^*), \theta^*) \in \cl(\mathcal 
		F_2)$ and that $R_0(\theta^*)$ is the infimum of Problem 
		\ref{prob:budget-constrained}. 
		Lemma \ref{lem:feasible} guarantees that $\theta^* \in \mathcal G_2$. Consider 
		any other point $\theta \in \mathcal G_2$, and note that $(R_0(\theta), 
		w^*(\theta), \theta) \in \cl(\mathcal F_2)$ as well, so that $R_0(\theta^*) \le 
		R_0(\theta)$. Therefore $\theta^*$ is a minimizer for 
		\eqref{prob:budget-constrained-explicit}, so it is an optimal budget-constrained 
		allocation. 
	\end{proof} 
	
	\begin{IEEEbiography}[{\includegraphics[width=1in,height=1.25in,clip,keepaspectratio]
		{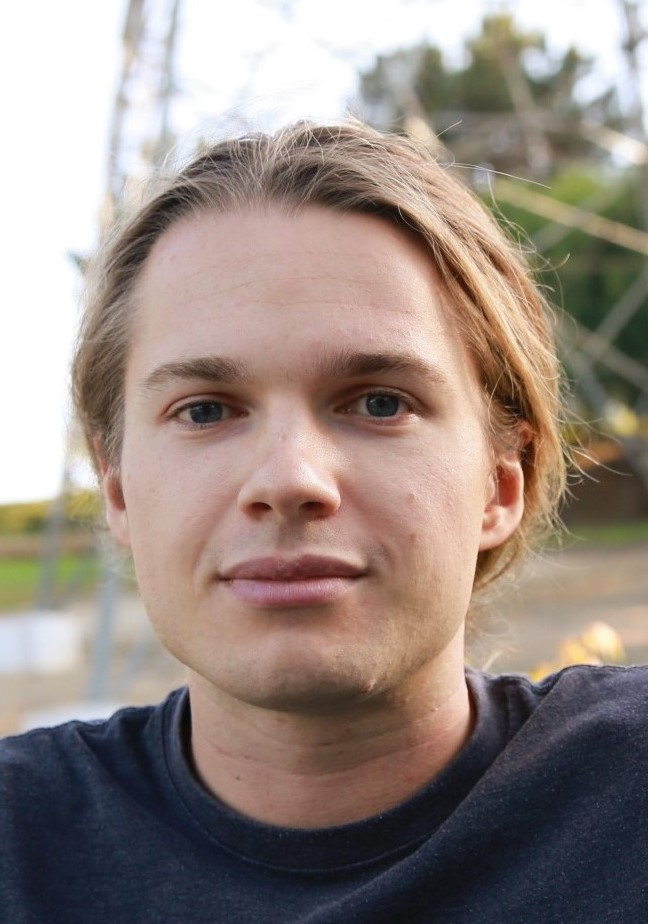}}]{Kevin D. Smith} is a Ph.D. candidate with the Center 
		for Control, Dynamical Systems and Computation at the University of California, 
		Santa Barbara (UCSB). He received his B.S. in physics from Harvey Mudd College in 
		2017 and his M.S. in electrical and computer engineering from UCSB in 2019. He is 
		interested in dynamics, control, and identification of network systems, including 
		power grids and other infrastructure systems.
	\end{IEEEbiography} 

	\begin{IEEEbiography}[{\includegraphics[width=1in,height=1.25in,clip,keepaspectratio]
		{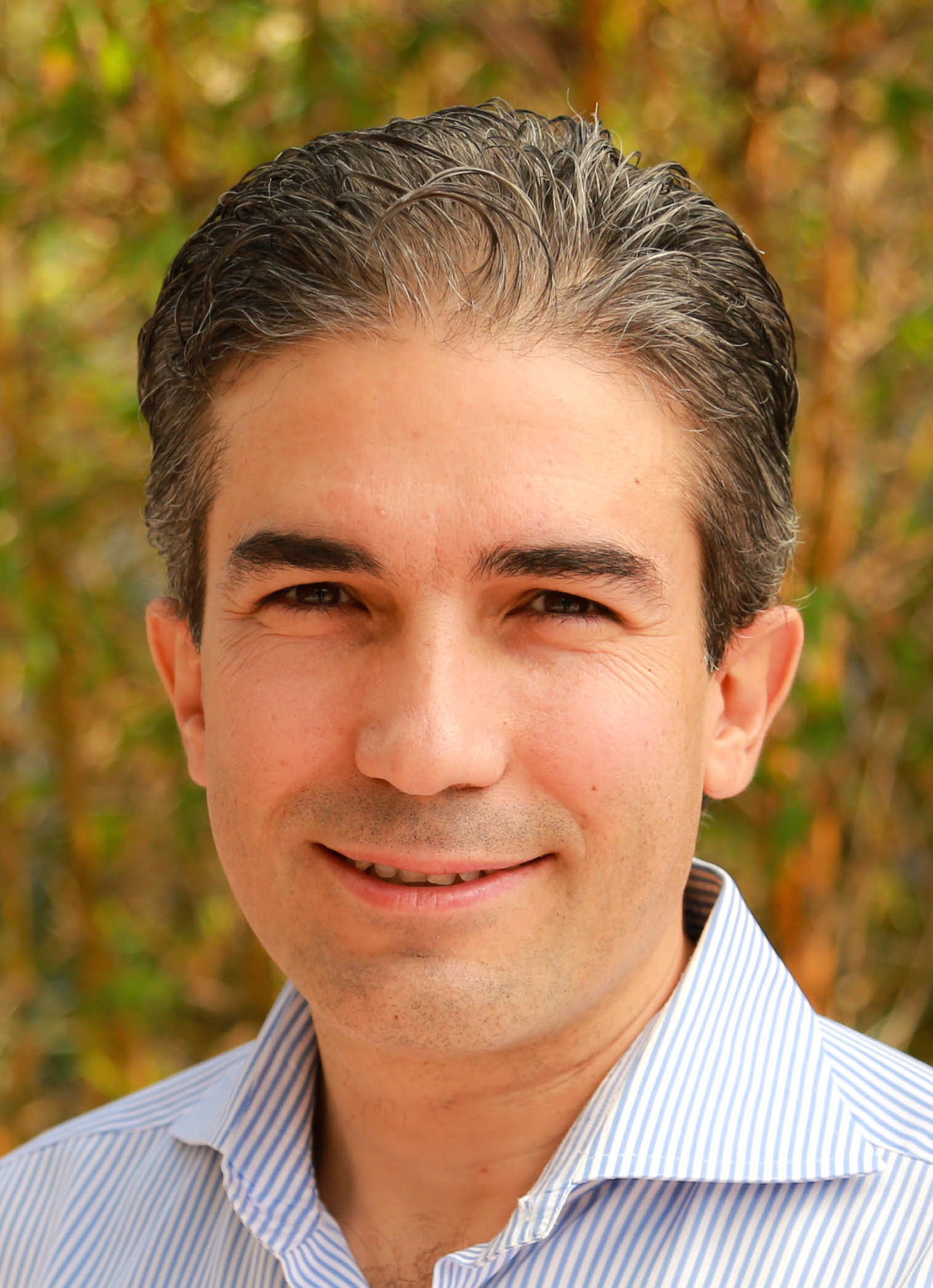}}]{Francesco Bullo} (Fellow, IEEE) is a
	      Distinguished Professor of Mechanical Engineering at the University
	      of California, Santa Barbara. He served as IEEE CSS President and as
	      SIAG CST Chair. His research focuses on contraction theory, network
	      systems and distributed control with application to machine
	      learning, power grids, social networks, and robotics. His latest book is 
	      "Contraction Theory for Dynamical Systems" (KDP, 2022, v1.0).  He is a 
	      Fellow of ASME, IFAC, and SIAM.
	\end{IEEEbiography}
	
	\bibliographystyle{ieeetr}
	\bibliography{alias,Main,FB}

\begin{thebibliography}{10}

\bibitem{RER-RL-CAG:09}
R.~E. Rowthorn, R.~Laxminarayan, and C.~A. Gilligan, ``Optimal control of
  epidemics in metapopulations,'' {\em Journal of the Royal Society Interface},
  vol.~6, no.~41, pp.~1135--1144, 2009.

\bibitem{SL-GC-CCC:10}
S.~Lee, G.~Chowell, and C.~Castillo-Ch\'avez, ``Optimal control for pandemic
  influenza: the role of limited antiviral treatment and isolation,'' {\em
  Journal of Theoretical Biology}, vol.~265, pp.~136--150, 2010.

\bibitem{MH-FB-VMP:21}
M.~Hayhoe, F.~Barreras, and V.~M. Preciado, ``Multitask learning and nonlinear
  optimal control of the {COVID-19} outbreak: A geometric programming
  approach,'' {\em Annual Reviews in Control}, 2021.

\bibitem{VLJS-IRM:21}
V.~L.~J. Somers and I.~R. Manchester, ``Sparse resource allocation for control
  of spreading processes via convex optimization,'' {\em IEEE Control Systems
  Letters}, vol.~5, no.~2, pp.~547--552, 2020.

\bibitem{VMP-MZ-CE-AJ-GJP:14}
V.~M. Preciado, M.~Zargham, C.~Enyioha, A.~Jadbabaie, and G.~J. Pappas,
  ``Optimal resource allocation for network protection against spreading
  processes,'' {\em IEEE Transactions on Control of Network Systems}, vol.~1,
  no.~1, pp.~99--108, 2014.

\bibitem{JAT-SR-YW:17}
J.~A. Torres, S.~Roy, and Y.~Wan, ``Sparse resource allocation for linear
  network spread dynamics,'' {\em IEEE Transactions on Automatic Control},
  vol.~62, no.~4, pp.~1714--1728, 2017.

\bibitem{CN-VMP-GJP:17}
C.~Nowzari, V.~M. Preciado, and G.~J. Pappas, ``Optimal resource allocation for
  control of networked epidemic models,'' {\em IEEE Transactions on Control of
  Network Systems}, vol.~4, pp.~159--169, 2017.

\bibitem{VSM-AB-KM:18}
V.~S. Mai, A.~Battou, and K.~Mills, ``Distributed algorithm for suppressing
  epidemic spread in networks,'' {\em IEEE Control Systems Letters}, vol.~2,
  no.~3, pp.~555--560, 2018.

\bibitem{ARH-JG-PEP:21}
A.~R. Hota, J.~Godbole, and P.~E. Par\'e, ``A closed-loop framework for
  inference, prediction, and control of {SIR} epidemics on networks,'' {\em
  IEEE Transactions on Network Science and Engineering}, vol.~8, no.~3,
  pp.~2262--2278, 2021.

\bibitem{PvdD-JW:02}
P.~V. den Driessche and J.~Watmough, ``Reproduction numbers and sub-threshold
  endemic equilibria for compartmental models of disease transmission,'' {\em
  Mathematical Biosciences}, vol.~180, no.~1, pp.~29--48, 2002.

\bibitem{OD-JAPH-JAJM:90}
O.~Diekmann, J.~A.~P. Heesterbeek, and J.~A.~J. Metz, ``On the definition and
  the computation of the basic reproduction ratio ${R}_0$ in models for
  infectious diseases in heterogeneous populations,'' {\em Journal of
  Mathematical Biology}, vol.~28, no.~4, pp.~365--382, 1990.

\bibitem{MO-MK-JL:20}
M.~{Ogura}, M.~{Kishida}, and J.~{Lam}, ``Geometric programming for optimal
  positive linear systems,'' {\em IEEE Transactions on Automatic Control},
  vol.~65, no.~11, pp.~4648--4663, 2020.

\bibitem{SB-SJK-LV-AH:07}
S.~Boyd, S.-J. Kim, L.~Vandenberghe, and A.~Hassibi, ``A tutorial on geometric
  programming,'' {\em Optimization and Engineering}, vol.~8, no.~1,
  pp.~67--127, 2007.

\bibitem{MG-SB:11-cvx}
M.~Grant and S.~Boyd, ``{CVX}: Matlab software for disciplined convex
  programming, version 2.1,'' Mar. 2014.

\bibitem{FB:22}
F.~Bullo, {\em Lectures on Network Systems}.
\newblock Kindle Direct Publishing, {1.6}~ed., Jan. 2022.

\bibitem{RAH-CRJ:12}
R.~A. Horn and C.~R. Johnson, {\em Matrix Analysis}.
\newblock Cambridge University Press, 2nd~ed., 2012.

\end{thebibliography}
		
\end{document}